\documentclass[psamsfonts, reqno]{amsart}
\usepackage{amscd, amsmath, amssymb, amsfonts, mathtools,hyperref}

\numberwithin{equation}{section}

\newtheorem{theorem}[equation]{Theorem}
\newtheorem{corollary}[equation]{Corollary}
\newtheorem{lemma}[equation]{Lemma}
\newtheorem{proposition}[equation]{Proposition}
\newtheorem{question}[equation]{Question}
\newtheorem*{theorem*}{Theorem}

\theoremstyle{remark}
\newtheorem{remark}[equation]{\bf Remark}
\theoremstyle{remark}
\newtheorem{example}[equation]{\bf Example}
\theoremstyle{remark}
\newtheorem{definition}[equation]{\bf Definition}

\newcommand{\A}{\mathbb{A}}

\newcommand{\C}{\mathbb{C}}

\newcommand{\tens}[1]{%
  \mathbin{\mathop{\otimes}\displaylimits_{#1}}%
}

\newcommand{\ML}{\mathrm{ML}}
\newcommand{\MLs}{\mathrm{ML^*}}
\newcommand{\Ufd}{\mathrm{UFD}}
\newcommand{\Exp}{\mathrm{EXP}}
\newcommand{\Exps}{\mathrm{EXP^*}}
\newcommand{\Tr}{\mathrm{tr.deg}}

\title{An algebraic characterization of the affine three space in arbitrary characteristic \hspace{3mm}}

\author[Sai Krishna]{P.M.S.Sai Krishna} 
\address{P.M.S.Sai Krishna, Department of Mathematics, Indian Institute of Technology Bombay, Powai,  Mumbai 400076,  India} 
\email{\href{mailto:saikrishna183@gmail.com}{saikrishna183@gmail.com},  \hspace{2mm}\href{mailto:204099001@iitb.ac.in}{204099001@iitb.ac.in}}

\begin{document}

\date{}
\maketitle
{\let\thefootnote\relax\footnotetext{{14R10,  13B25,  13A50,  13N15,  14R20}}

{\let\thefootnote\relax\footnotetext{{Keywords:}~ {polynomial ring,  exponential map,  Makar-Limanov invariant}}}

\begin{abstract}
    We give an algebraic characterization of the affine $3$-space over an algebraically closed field of arbitrary characteristic. We use this characterization to reformulate the following question. Let $$A=k[X, Y, Z, T]/(XY+Z^{p^e}+T+T^{sp})$$ where $p^e\nmid sp$,  $sp\nmid p^e$,  $e, s\geq 1$ and $k$ is an algebraically closed field of positive characteristic $p$. Is $A= k^{[3]}?$ We prove some results on $\ML$ and $\MLs$ invariants and use them to prove a special case of the strong cancellation of $k^{[2]}$.
\end{abstract}
 
\section{Introduction}
\emph{
Throughout the paper, $k$ is a field,  rings are commutative with unity,  $R^{[n]}$ denotes the polynomial ring in $n$ variables over the ring $R$,  $\Ufd$ denotes a Unique Factorization Domain,  $B$ is a $k$-domain,  $\Tr _kB$  is the transcendence degree of the fraction field of $B$ over $k$ and $B^*$ is the group of units of $B$.}

Exponential maps are also known as locally finite iterative higher derivations $(\mathrm{lfihd}).$ Working with locally nilpotent derivations of $k$-domains is not usually possible when $k$ has positive characteristic. Exponential maps generalize locally nilpotent derivations and capture information of the domain's higher-order derivations. We note that when $k$ is of characteristic zero,  exponential maps are equivalent to locally nilpotent derivations. 

The characterization of polynomial rings is an important problem in Affine Algebraic Geometry. For instance,  a direct application of characterization of $k^{[2]}$ as stated in (\ref{miya}) will prove that $k^{[2]}$ is cancellative for any algebraically closed field of arbitrary characteristic. We note that in the case of zero characteristic, the cancellation of $k^{[2]}$ for an arbitrary field can be deduced from the case when $k$ is algebraically closed by Kambayashi's Theorem \cite{K}, which is stated in \cite[Theorem $5.2$]{Gene}. In \cite{Rus}, it was shown that $k^{[2]}$ is cancellative when $k$ is a perfect field of arbitrary characteristic. In the case of positive characteristic,  the cancellation of $k^{[2]}$ for any field was shown in \cite{BN} and \cite{Koj}. We give an alternate proof (\ref{cancel2}) of the cancellation of $k^{[2]}$ for any field $k$ using a characterization (\ref{miya2}) of $k^{[2]}$.

The $\ML$-invariant was introduced by Lenoid Makar-Limanov in \cite{ML} to show that the Koras-Russel threefold $x+x^2y+z^2+t^3$ over $\C$ is not isomorphic to $\C ^3.$ The $\MLs$-invariant was introduced in \cite[page 237]{GF1}.  In \cite[Theorem $1$]{GS}, it was shown that when $k$ is an algebraically closed field of zero characteristic and if $\MLs $ is non-trivial, then the $\MLs $-invariant coincides with the $\ML $-invariant for affine $k$-domains. A natural question is whether the same result holds when $k$ is of arbitrary characteristic. Lemmas (\ref{tr2}) and $(\ref{tr3 equality})$ provide partial answers to this question.

The main result of this paper is the following algebraic characterization of $k^{[3]}$, which is proved in \cite[Theorem $4.6$]{NN} when $k$ is of zero characteristic. We give a characteristic free proof (\ref{affine 3space}).

\begin{theorem}
        Let $k$ be an algebraically closed field and $B$ be an affine $\Ufd $ over $k$ with $\Tr _kB=3$. Then the following are equivalent. 
    \begin{enumerate}
        \item  $B= k^{[3]}$

        \item $\ML ^*(B)=k$

        \item $\ML (B)=k$ and $\MLs (B)\neq B$.
    \end{enumerate}
\end{theorem}

We show a possible application of the above characterization in the following situation. In \cite{Asa1},  Asanuma introduced the family of rings  $$A=k[X, Y, Z, T]/(X^mY+Z^{p^e}+T+T^{sp})$$  where $k$ is of positive characteristic $p$ ,  $p^e\nmid sp, sp\nmid p^e$,   $m, e, s\geq 1$ and showed that $A^{[1]} = k^{[4]}$. In \cite{NG1}, Gupta showed that $A\neq k^{[3]}$ when $m>1$, and thus resolved the Zariski cancellation problem for $k^{[3]}$ in the case of positive characteristic. In \cite{NG2}, the Zariski cancellation problem was completely solved by Gupta in the positive characteristic case. However, it is not known whether $A= k^{[3]}$ when $m=1$. As mentioned in \cite[Remark $2.3$]{Asa2},  if $A= k^{[3]}$, then it will give an example of non-linearizable torus action on $k^{[3]}$ in positive characteristic which will serve as a counter-example to the linearization problem (which is open for $k^{[3]}$ in the case of positive characteristic). If $A\neq k^{[3]}$,  then we will get another counter example to the Zariski cancellation problem for $k^{[3]}$. 

The brief outline of the paper is as follows. In section $2$,  we state a few basic properties of exponential maps and introduce the Makar-Limanov $(\ML )$-invariant and Makar-Limanov-Freduenburg $(\MLs )$-invariant,  which are used in our characterization of $k^{[3]}$. In section $3$, we give an alternative proof of the algebraic characterization \cite{Miya} of $k^{[2]}$ over an algebraically closed field. In the case of zero characteristic,  a proof of this result is presented in \cite[Theorem $9.12$]{Gene}. Using exponential maps,  we give a characteristic free proof (\ref{miya}) of \cite[Theorem $1$]{Miya}. We extend this result to arbitrary fields (\ref{miya2}). Next, in section $3$,  we give a characteristic free proof  (\ref{2space}) of the algebraic characterization of $k^{[2]}$ presented in \cite[Theorem $3.8$]{NN} for any field of zero characteristic. In section $4$, we prove the main result of this paper(\ref{affine 3space}). In section $5$,  we state some properties of $\ML$ and $\MLs $ invariants, some of which are mentioned in \cite{NN} in the case of zero characteristic. The stability results for $\ML$ and $\MLs$ invariants are particularly interesting. It is a well known result that for a $k$-domain $B$ with $\Tr _kB=1$,  we have that $\ML(B^{[n]})=\ML(B)$. We extend this result to domains of transcendence degree $2$ under certain assumptions and using them we prove a special case the strong cancellation of $k^{[2]}$ under some hypothesis. We prove similar results for $\MLs$-invariant and use them to prove some results that partially answer the question (\ref{quest}). 

\section{Preliminaries}
We recall some definitions.
\begin{definition}
    Let $B$ be a $k$-domain. Let $\delta: B\longrightarrow B^{[1]}$ be a $k$-algebra homomorphism. We denote it by $\delta_t:B\longrightarrow B[t]$ if we want to emphasize the indeterminate. We call $\delta$ an \emph{exponential map} on $B$ if
    \begin{enumerate}
        \item $\epsilon_0\delta_t$ is identity on $B$ where $\epsilon_0:B[t]\longrightarrow B$ is the evaluation map at $t=0.$

        \item $\delta_s\circ \delta_t=\delta_{s+t}$,  where $\delta_s$ is extended to a homomorphism $B[t]\longrightarrow B[s, t]$ by defining $\delta_s(t)=t$.
    \end{enumerate}

We denote the set of all exponential maps on $B$ by $\Exp (B)$. The set $B^{\delta}=\{x\in B| \delta(x)=x\}$ is called the \emph{ring of $\delta$-invariants} of $B$.
\end{definition}
\begin{example}\label{exam}
    The inclusion map $B \xhookrightarrow{}B^{[1]}$ is a $\emph{trivial}$ exponential map on $B$. Let $B=k[X_1, \dots,  X_n]$ be the polynomial ring in $n$ variables. Let $\delta_i \in \Exp (B)$ be such that $\delta_i(X_i)=X_i+t$ and $\delta_i(X_j)=X_j$ $\forall \ j\neq i$. $\delta_i$ is called a \emph{shift} exponential map on the polynomial ring. 
\end{example}

\begin{remark}\label{remark1}
Let $\delta$ be an exponential map on a $k$-domain $B$.
\begin{enumerate}
    \item     We can express $\delta$ as follows. $$\delta_t(x)=D_0(x)+D_1(x)t+D_2(x)t^2+\cdots+D_m(x)t^m$$ where $D_0, D_1, D_2, \dots $ are iterative derivatives associated with $\delta$ and $m=\deg_t(\delta(x))$. From property $(1)$ of exponential maps,  we have that $D_0$ is the identity map. Note that $B^{\delta}=\bigcap _{i\geq 1}\mathrm{Ker}(D_i)$.

    \item For any $b\in B$,  we can define the \emph{$\delta$}-degree of a nonzero $b$ as the $t$-degree of $\delta(b)$ i.e.  $\deg_{\delta}(b)=\deg_t(\delta(b))$. We define $\deg_{\delta}(0)=-\infty. $

    \item The \emph{$\delta$}-degree function is a degree function satisfying the following properties. 
    \begin{enumerate}
        \item $\deg_{\delta}(ab)=\deg_{\delta}(a)+\deg_{\delta}(b)$ for all $a, b \in B$
        
        \item $\deg_{\delta}(a+b)\leq \max\{ \deg_{\delta}(a), \deg_{\delta}(b)\}$ and the equality holds when $ \deg_{\delta}(a)\neq \deg_{\delta}(b)$.

    \end{enumerate}

    \item Consider $\epsilon_\lambda: B^{[1]} \longrightarrow B$ which is the evaluation map at $\lambda \in k$. We note that $\epsilon_\lambda \circ \delta$ is a $k$-algebra automorphism of $B$ with inverse $\epsilon_{-\lambda}\circ \delta$.
    
\end{enumerate}
\end{remark}

\begin{definition}
Let $\delta$ be a non-trivial exponential map on a $k$-domain $B$. 
\begin{enumerate}
    \item  Any $x\in B$ such that $\deg_{\delta}(x)=\min\{\deg_{\delta}(a)|a\in B\setminus B^{\delta}\}$ is called a \emph{local slice} of $\delta$. Every non-trivial exponential map has a local slice.

    \item An element $x\in B$  is called a \emph{slice} if $x$ is a local slice of $\delta$ and $D_n(x)=1$  or equivalently $D_n(x)$ is a unit,  where $n=\deg_{\delta}(x)$. It is not necessary that every non-trivial exponential map has a slice. We denote the set of all exponential maps on $B$ which have a slice by $\Exps (B).$

    \item Any non-trivial $\delta\in \Exp (B)$ is said to be a \emph{reducible} exponential map if there exist a proper principal ideal $I$ of $B$ such that $D_i(B)\subseteq I^i$ for all $i\geq 1$. If a non-trivial $\delta \in \Exp (B)$ is not reducible, then it is called an \emph{irreducible} exponential map.

    \item Let $\delta_1, \delta_2 \in \Exp (B)$. They are said to be \emph{equivalent} if $B^{\delta_1}=B^{\delta_2}$.

    \item The \emph{Makar-Limanov invariant} is a $k$-subalgebra of $B$ defined as $\ML (B)=\cap _{\delta \in \Exp (B)}B^{\delta}.$
    
    \item The \emph{Makar-Limanov-Freudenburg invariant} is a $k$-subalgebra of $B$ defined as 
    
    $\MLs (B)=\cap _{\delta \in \Exps (B)}B^{\delta}.$ If $\Exps (B)=\phi$,  then we define $\MLs (B)=B$. 

    \item $B$ is called \emph{rigid} if there does not exists any non-trivial exponential map on $B$. This is equivalent to the condition $\ML (B)=B$.

    \item $B$ is $\emph{semi-rigid}$  if there exists  $\delta \in \Exp (B)$ such that $\ML (B)=B^{\delta}.$ A rigid domain is also semi-rigid. 

    \item $B$ is called \emph{geometrically factorial} if $B\tens{k} L$ is a $\Ufd$, where $L$ is any algebraic extension of $k$.   

    \item A subring $A$ of $B$ is said to be \emph{factorially closed} in $B$ if for any non-zero $x, y \in B$ such that $xy\in A$,  then $x, y \in A$. Note that if $A$ is factorially closed in $B$, then $A$ is algebraically closed in $B$.
 
\end{enumerate}   
\end{definition}

We summarise below some useful properties of exponential maps \cite[Lemma $2.1,  2.2$]{CM}.

\begin{remark}\label{first principles}
Let $\delta$ be a non-trivial exponential map on a $k$-domain $B$ and $x$ be a local slice with $n=\deg_{\delta}(x)$. Let $c=D_n(x)$.
\begin{enumerate}
    \item $D_i(x)\in B^{\delta}$ for all $i>0.$ 
    
    \item  $B^{\delta}$ is factorially closed in $B$. In particular,  $B^{\delta}$ is algebraically closed in $B$. 

    \item $B[c^{-1}]=B^{\delta}[c^{-1}][x]$.
    
    \item  $\Tr _kB^{\delta}=\Tr _kB-1$.

    \item Intersection of factorially closed rings is factorially closed. Hence $\ML (B)$ and $\MLs (B)$ are factorially closed in $B$.
    
    \item $k \subseteq \ML (B)\subseteq \MLs (B)$ and hence if $\MLs (B)=k$,  then $\ML (B)=k$. 

    \item $\ML (k^{[n]})=\MLs (k^{[n]})=k$ for all $n\geq 1$.

    \item  $\delta$ can be extended to an exponential map on $B^{[n]}=B[X_1, \dots, X_n]$ by fixing $X_i$. The ring of invariants of this extended exponential map on $B[X_1, \dots, X_n]$ is $B^{\delta}[X_1, \dots, X_n]$. This, combined with the fact that there always exists shift exponential maps, gives that $\ML (B^{[n]})\subseteq \ML (B)$. By the same argument,  it follows that $\MLs  (B^{[n]})\subseteq \MLs (B)$.

    \item If $B$ is a semi-rigid $k$-domain which is not rigid such that $\Tr _kB<\infty$,  then for any non-trivial $\epsilon \in \Exp (B)$,  we have that $B^{\epsilon}=\ML (B)$. This follows since $\ML (B)=B^{\delta}$ for some $\delta \in \Exp (B)$ and  $\ML (B)=B^{\delta}\subseteq B^{\epsilon}$. $\ML (B)$ and $B^{\delta}$ are algebraically closed in $B$ and have the same transcendence degree and hence $\ML (B)=B^{\epsilon}$. Thus,  all non-trivial exponential maps of $B$ have the same ring of invariants. 

    \item A factorially closed subring of a $\Ufd$ is a $\Ufd$.
\end{enumerate}
\end{remark} 

\noindent
The third property in the above remark is quite useful. We record an important special case of it as the  Slice Theorem. In the case of zero characteristic this coincides with \cite[Corollary $1.26$]{Gene}.

\begin{corollary}[\emph{Slice Theorem}]\label{slice theorem}
Let $\delta$ be a non-trivial exponential map on a $k$-domain $B$ and $x$ be a slice with $n=\deg_{\delta}(x)$. Then  $B=B^{\delta}[x]$ and since $x$ is transcendental over $B^{\delta}$,  it follows that $B= (B^{\delta}) ^{[1]}.$
\end{corollary}
    
We recall the following result from \cite[Lemma $2.3$]{CML}.
\begin{lemma}\label{tr1}
Let $B$ be a $k$-domain with $\Tr _kB=1$ and $k$ is algebraically closed in $B$.  Then $\ML (B)=k$ if and only if $B= k^{[1]}$. Otherwise,  $\ML (B)=B$.    
\end{lemma}

\begin{theorem}[\emph{Semi-Rigidity Theorem}]\label{semi}
Let $B$ be a domain which is either an affine $k$-domain or finitely generated as a ring. Then $B$ is rigid if and only if $B^{[1]}$ is semi-rigid.  Equivalently,  $B \textit{ is rigid}$ if and only if $\ML (B^{[1]})=B. $       
\end{theorem}

\begin{proof} Clearly if $\ML (B^{[1]})=B$,  then by $(8)$ of (\ref{first principles}),  it follows that $B$ must be rigid. We refer to \cite[Theorem $3.1$]{CM} for the other direction.
\end{proof}

We recall the following two results \cite[Theorem $3.1$]{Crac} and \cite[Corollary $3.3$]{SK}.
\begin{theorem}\label{ufd slice}
    Let $B$ be a $\Ufd $  over an algebraically closed field $k$ with $\Tr _kB=2$ and a non-trivial $\delta \in \Exp (B)$. Then $B= (B^{\delta})^{[1]}.$ 
\end{theorem}

\begin{corollary}\label{gufd slice}
      Let $B$ be a $\Ufd$ over a field $k$ with $\Tr _kB=2$ and a non-trivial $\delta \in \Exp (B)$. Suppose $B$ is geometrically factorial. Then $B= (B^{\delta})^{[1]}.$
\end{corollary}

\begin{remark}
Let $B$ and $k$ be as in (\ref{ufd slice})
   \begin{enumerate}
       \item In the case of characteristic zero,  (\ref{ufd slice}) coincides with \cite[Lemma $2.10$]{Gene}. However, the above theorem does not state that every irreducible exponential map on $B$ has a slice. It states that every non-trivial exponential map on $B$ is equivalent to another exponential map which has a slice.

       \item  The following is an example of an irreducible exponential map that doesn't have a slice. Consider $\delta \in \Exp (k[x, y])$ defined by $\delta(x)=x$ and $\delta(y)=y+t+x^2t^p$,  where $k$ is of characteristic $p$.  Notice that for any $f=\sum_{i=0}^n a_i(x)y^i$  where $a_i(x)\in k[x]$,  we have that $\delta(f)=\sum_{i=0}^na_i(x) (y+t+x^2t^p)^i$. The leading coefficient of $\delta(f)$ is $a_n(x)x^{2n}$, and it cannot be a unit; thus $\delta$ has no slice. Moreover,  $\delta$ is an irreducible exponential map. 
       
   \end{enumerate}

\end{remark}

\noindent
The following proposition \cite[Lemma $2.9$,  Remark $2.14$]{Gene} is a characterization of an affine $\Ufd$ of transcendence degree $1$. 
\begin{proposition}\label{ufd tr1}
        Let $B$ be an affine $\Ufd $ over an algebraically closed field $k$ with $\Tr _kB=1$. Then $B= k[t]_{f}$ for some $f\in k[t]$. Moreover,  if $B^*=k^*$,  then $B = k^{[1]}$.
\end{proposition}

\section{Characterization of $k^{[2]}$}

The following algebraic characterization of $k^{[2]}$ is due to Miyanishi \cite[Theorem $1$]{Miya}. We give a proof using exponential maps. 

\begin{theorem}\label{miya}
    Let $k$ be an algebraically closed field and $B$ be an affine $\Ufd$ over $k$ with $\Tr _kB=2$. Then the following are equivalent. 
 \begin{enumerate}
     \item $B= k^{[2]}$.

     \item $\ML (B)=k$
     
     \item $B$ is not rigid and $B^*=k^*$.
   
 \end{enumerate}
\end{theorem}

\begin{proof}
    Clearly ${1}\Rightarrow{2} \Rightarrow 3$. We show $3 \Rightarrow 1$. Since $B$ is not rigid, by (\ref{ufd slice}), $B=A^{[1]}$ where $A=B^{\delta}$ for some $\delta \in \Exps (B)$  and $\Tr _kA=1$. By $(10)$ of \ref{first principles}, $A$ is a $\Ufd$. Note that   $A^*=k^*$ and $A$ is affine. By (\ref{ufd tr1}),  we have that $A= k^{[1]}$ and hence $B = k^{[2]}$.  
\end{proof}

We can slightly modify the above theorem by removing the condition that $k$ is algebraically closed and requiring that $B$ is geometrically factorial. 

\begin{theorem}\label{miya2}
    Let $B$ be an affine $k$-domain of $\Tr _kB=2$, which is geometrically factorial. Then, the following are equivalent. 
 \begin{enumerate}
     \item $B= k^{[2]}$.

     \item $\ML (B)=k$

 \end{enumerate}
\end{theorem}

\begin{proof}
 Clearly, $1 \Rightarrow 2$. We show $2 \Rightarrow 1$. Suppose $\ML (B)=k$. Then $B$ is not rigid. Consider a non-trivial $\delta \in \Exp (B)$. Then by (\ref{gufd slice}),  $B=A^{[1]}$ where $A=B^{\delta}$ and $\Tr _kA=1$. We have that $\ML (B)=k$ is algebraically closed in $B$ and hence in $A$. By  Semi-Rigidity Theorem (\ref{semi}),  $A$ is not rigid. By (\ref{tr1}),  $A = k^{[1]}$ and hence $B = k^{[2]}$.
\end{proof}

The following characterization of $k^{[2]}$ is presented in \cite[Theorem $3.8$]{NN} for a field of characteristic zero. We give a characteristic free proof of the same result. 
\begin{theorem}\label{2space}
    Let $B$ be an affine $k$-domain with $\Tr _kB=2$. Then the following are equivalent. 
    \begin{enumerate}
        \item $B= k^{[2]}$

        \item $\MLs (B)=k$

        \item $\ML (B)=k$ and $\MLs (B)\neq B$.
    \end{enumerate}
\end{theorem}

\begin{proof}
Clearly $1 \Rightarrow 2 \Rightarrow 3$.  We show $3\Rightarrow 1$. Since $\MLs (B)\neq B$,  so there exists a non-trivial $\delta \in \Exp (B)$ which has a slice. By Slice Theorem (\ref{slice theorem}),  we have that $B=A^{[1]}$ where $A=B^{\delta}$. By (\ref{first principles}),  we have that $\Tr _kA=1$. Since $\ML (B)$ is algebraically closed (in fact, factorially closed) in $B$,  it follows that $k$ is algebraically closed in $A$. Since $\ML (B)=k$,  so by Semi-Rigidity Theorem (\ref{semi}), $A$ is not rigid. It follows from (\ref{tr1}) that $A= k^{[1]}$ and hence $B = k^{[2]}$. 
\end{proof}

\section{Characterization of $k^{[3]}$}
The following characterization of $k^{[3]}$ is presented in \cite[Theorem $4.6$]{NN} for a field of zero characteristic. We give a characteristic free proof of the same result.

\begin{theorem}\label{affine 3space}
        Let $k$ be an algebraically closed  field and $B$ be an affine $\Ufd $ over $k$ with $\Tr _kB=3$. Then the following are equivalent. 
    \begin{enumerate}
        \item  $B= k^{[3]}$

        \item $\MLs (B)=k$

        \item $\ML (B)=k$ and $\MLs (B)\neq B$.
    \end{enumerate}
\end{theorem}

\begin{proof}
Clearly $1 \Rightarrow 2 \Rightarrow 3$. We show $3 \Rightarrow 1$. Since $\MLs (B)\neq B$,  so there exists a non-trivial $\delta \in \Exps (B)$. By Slice Theorem (\ref{slice theorem}),  we have that $B=A^{[1]}$ where $A=B^{\delta}$ and $\Tr _kA=2$. By $(10)$ of \ref{first principles}, $A$ is a $\Ufd$. Since $\ML (B)=k$ and $B=A^{[1]}$,  it follows from Semi-Rigidity Theorem (\ref{semi}) that $A$ is not rigid. $\ML (B)=k$ implies that $B^*=k^*$ and hence $A^*=k^*$.  Now,  we have that $A$ is an affine $\Ufd$ over $k$ with $\Tr _kA=2$, $A$ is not rigid and $A^*=k^*$. By (\ref{miya}),  $A = k^{[2]}$ and hence $B = k^{[3]}$.  
\end{proof}

Let $G$ be a reductive algebraic group acting on affine space $\A^n$  over a field $k$. The linearization conjecture as mentioned in \cite{Asa1}  states that the action of a reductive algebraic group is linear after a suitable polynomial change of the coordinate system on $\A^n$. In the case of positive characteristic, it was shown to be false for $n\geq 4$  in \cite{Asa2}. However,  it is not known for $n=3$. Consider the ring $$A=k[X, Y, Z, T]/(XY+Z^{p^e}+T+T^{sp})$$ where $k$ is a field of positive characteristic $p$  such that $p^e\nmid sp$, $sp\nmid p^e$,  and $m, e, s\geq 1$. It was shown in \cite{Asa1} that $A^{[1]}= k^{[4]}$. If $A= k^{[3]}$,  then $A$ is a counter-example to the linearization conjecture.

\begin{remark}\label{app}
  By \cite[Lemma $3.1$]{NG3} and using that $f(Z, T)=Z^{p^e}+T+T^{sp}$ is irreducible in $k[Z, T]$,  we can conclude that $A$ is a $\Ufd$. By $2$ of \cite[Remark $4.7$]{NG3},  we know that $\ML (A)=k$. Hence,  we can reduce the question about whether $A$ is isomorphic to $k^{[3]}$ as follows.    
\end{remark}

\begin{corollary}
    Let $k$ be an algebraically closed field of positive characteristic and $A=k[X, Y, Z, T]/(XY+Z^{p^e}+T+T^{sp})$. Then the following are equivalent. 
    \begin{enumerate}
        \item $A= k^{[3]}$.

        \item $\MLs (A)\neq A$. This is equivalent to the existence of an exponential map on $A$ that has a slice.
    \end{enumerate}
\end{corollary}

\begin{proof}
    This follows from (\ref{affine 3space}) and (\ref{app}).
\end{proof}

\begin{remark}
    The above corollary also follows directly from \cite[Theorem 4]{Rus}. 
\end{remark}

We can slightly modify (\ref{affine 3space}) by removing the condition that $k$ is algebraically closed and requiring that $B$ is a geometrically factorial and $k$ is an infinite field. We prove the following theorem in section $5$ after corollary (\ref{temp}).

\begin{theorem}\label{3 space}
    Let $B$ be an affine $\Ufd $ over $k$ with $\Tr _kB=3$, where $k$ is an infinite field. Suppose $B$ is geometrically factorial. Then the following are equivalent. 
    \begin{enumerate}
        \item  $B= k^{[3]}$

        \item $\MLs (B)=k$

        \item $\ML (B)=k$ and $\MLs (B)\neq B$.
    \end{enumerate}
\end{theorem}

\section{Some Properties of $\ML$ and $\MLs $  invariants}
\noindent
In this section, we first describe some properties related to $\ML$ and $\MLs$ invariants and look at some results relating the two invariants.

\begin{lemma}\label{ML*}
    Let $B$ be an affine $k$-domain. Then the following hold.
    \begin{enumerate}
        \item $\MLs (B^{[n]})\subseteq \ML (B)$

        \item $B$ is rigid $\Leftrightarrow$ $\MLs (B^{[1]})=B$.
    \end{enumerate}
    
\end{lemma}

\begin{proof}
    \textit{(1)} Let $B[X_1, \dots, X_n]=B^{[n]}$. Since there are shift exponential maps on $B^{[n]}$,  we have $\MLs (B^{[n]})\subseteq B$. Suppose  $\delta \in \Exp (B)$,  we can extend it to $\Tilde{\delta}_i\in \Exps (B[X_1, \dots, X_n])$ by defining $\Tilde{\delta}_i(X_j)=X_j \ \forall \  j\neq i$ and $\Tilde{\delta}_i(X_i)=X_i+t$. Now suppose $b \in \MLs (B^{[n]})\subseteq B$,  then $\Tilde{\delta_i}(b)=b$ and hence $\delta(b)=b$ for any $\delta \in \Exp (B).$   It follows that $\MLs (B^{[n]})\subseteq \ML (B)$.

    \textit{(2)} If $\MLs (B^{[1]})=B$,  then by previous part it follows that $\ML (B)=B$ and hence $B$ is rigid. Now,  suppose $B$ is rigid. By Semi-Rigidity Theorem (\ref{semi}),  $\ML (B^{[1]})=B$.  Since $\ML (B^{[1]})\subseteq \MLs (B^{[1]})$,  thus it follows that $B\subseteq \MLs (B^{[1]})$ and again from previous part we can conclude that $\MLs (B^{[1]})=B$.   
\end{proof}

The following result is about the stability of the $\ML$ invariant for $k$-domains of transcendence degree $1$. This is proved in \cite[Theorem $2.28$]{Gene} when $k$ is of zero characteristic. We adapt the same proof to prove the result for arbitrary characteristic. When $k$ is algebraically closed,  a proof of the following Theorem can be found in \cite{CML}.

\begin{theorem}\label{stab}
     Let $B$ be a $k$-domain with $\Tr _kB=1$, where $k$ is an infinite field. Then $\ML (B^{[n]})=\ML (B)$ for every integer $n\geq 0$.
\end{theorem}

\begin{proof}
    If $B$ is not rigid, then by $(2)$ of (\ref{first principles}),  $B=L^{[1]}$ where $L$ is the algebraic closure of $k$ in $B$. Then,  clearly $\ML (B^{[n]})=\ML (B)=L$. If $B$ is rigid,  then it is not of the form $L^{[1]}$ for any algebraic extension of $k$. Using \cite[Theorem $3.3$]{AEH},  we get that $\phi(B)=B$ for any $k$-algebra automorphism of $B^{[n]}$. By $(4)$ of (\ref{remark1}),  for  any $\delta\in \Exp (B^{[n]})$ and $\lambda \in k$,  we have that $\epsilon_{\lambda} \circ \delta$ is a $k$-automorphism of $B^{[n]}$. Hence we have that $\epsilon_\lambda \circ \delta(B)=B$ for all $\lambda \in k$. Since $k$ is an infinite field, it follows that $\delta$  restricts to $B$. As $B$ is rigid, thus $\delta$ is identity on $B$. It follows that $B\subseteq \ML (B^{[n]})$ and hence $\ML (B^{[n]})=B=\ML (B)$.
\end{proof}

\begin{proposition}\label{stab2}
    Let $B$ be a $\Ufd$ over a field $k$ with $\Tr _kB=2$, where $k$ is an infinite field. Suppose $B$ is geometrically factorial and $B$ is not rigid, then $\ML (B^{[n]})=\ML (B)$. 
\end{proposition}

\begin{proof}
    Since $B$ is not rigid. Then by (\ref{gufd slice}),  we have that $B=A^{[1]}$ where $A=B^{\delta}\ $ for some non-trivial $\delta \in \Exp (B)$. We have that $\ML (B^{[n]})=\ML (A^{[n+1]})$ and $\ML (B)=\ML (A^{[1]})$. By (\ref{stab}), $\ML (A^{[m]})=\ML (A)$ for all $m\geq 0$ and the result follows. 
\end{proof}

\begin{corollary}\label{temp}
    Let $B$ be a $\Ufd $ over an algebraically closed field $k$ such that $\Tr _kB=2$. Suppose $B$ is not rigid, then $\ML (B^{[n]})=\ML (B)$.
\end{corollary}

We now prove (\ref{3 space}) using (\ref{stab2}).
\begin{proof}[\textbf{Proof of Theorem 4.5}]
Clearly $1 \Rightarrow 2 \Rightarrow 3$. We show $3 \Rightarrow 1$. Since $\MLs (B)\neq B$,  so there exists a non-trivial $\delta \in \Exps (B)$. By Slice Theorem (\ref{slice theorem}), $B=A^{[1]}$ where $A=B^{\delta}$ and $\Tr _kA=2$. $B=A^{[1]}$ and $B$ is geometrically factorial implies that $A$ is geometrically factorial. Since $\ML (B)=k$ and $B=A^{[1]}$,  it follows from Semi-Rigidity Theorem (\ref{semi}) that $A$ is not rigid. By \ref{stab2},we get that $\ML (A)=k$.  By \ref{miya2}, $A = k^{[2]}$ and hence $B = k^{[3]}$.  
\end{proof}

\begin{corollary}[\emph{Cancellation of $k^{[2]}$}]\label{cancel2}
    Let $B$ be a $k$-domain such that $B^{[1]}= k^{[3]}$. Then $B= k^{[2]}$.
\end{corollary}

\begin{proof}
We first prove the result, assuming that $k$ is an infinite field. Since $B^{[1]}= k^{[3]}$,  it follows that $B^{[1]}$ is geometrically factorial and hence $B$ is geometrically factorial. By Semi-Rigidity Theorem (\ref{semi}),  $B$ is not rigid. By (\ref{stab2}),  $ML(B)=ML(B^{[1]})=k$ and it follows by (\ref{miya2}) that $B= k^{[2]}$.

When $k$ is a finite field, then since finite fields are perfect, the result follows by Kambayashi's Theorem \cite{K} stated in \cite[Theorem $5.2$]{Gene}.
\end{proof}

The strong cancellation of $k^{[2]}$ states that if $B^{[n]}=k^{[n+2]}$ for some integer $n\geq 2$, then $B=k^{[2]}$. This was proved in \cite[Theorem 4]{Rus} for any perfect field $k$. We prove if for any field $k$ under some additional conditions.  
\begin{corollary}\label{strong}
    Let $B$ be a $k$-domain such that $B^{[n]}=k^{[n+2]}$ for some integer $n\geq 2$. Suppose $B$ is not rigid. Then $B=k^{[2]}$.
\end{corollary}

\begin{proof}
    If $k$ is a finite field, the result follows from \cite[Theorem 4]{Rus}. Since $B^{[n]}=k^{[n+2]}$, it follows that $B$ is geometrically factorial. Since $B$ is not rigid, by (\ref{stab2}), $\ML (B)=\ML (B^{[2]}=k$. It follows by (\ref{miya2}) that $B=k^{[2]}$.
\end{proof}

The above result raises the following question.

\begin{question}
    Let $B$ be a $k$-domain such that $B^{[n]}=k^{[n+2]}$ for some positive integer $n$. Is $B$ not rigid?
\end{question}

\begin{remark}
     By (\ref{strong}), we note that a positive answer to the above question will lead to strong cancellation of $k^{[2]}$ for any field $k$.
\end{remark}

\begin{corollary}
        Let $B$ be an affine $k$-domain. Suppose there exists $\delta \in \Exps (B)$ such that $B^{\delta}$ is rigid, then $\ML (B)=\MLs (B)$.
\end{corollary}
\begin{proof}
     Let $A=B^{\delta}$. By Slice Theorem (\ref{slice theorem}), $B=A^{[1]}$ and by Semi-Rigidity Theorem (\ref{semi}) and (\ref{ML*}), we get $\ML (B)=A$ and $\MLs (B)=A$ and hence $\ML (B)=\MLs (B)$.
\end{proof}

\begin{lemma}\label{deg1}
    Let $B$ be $k$-domain with $\Tr _kB=1$. Then $\ML (B)=\MLs (B)$.
\end{lemma}

\begin{proof}
    
If $B$ is rigid, then the result follows by definition.

Suppose $B$ is not rigid. We show that every exponential map on $B$ has a slice. Consider a non-trivial $\delta \in \Exp (B)$. Then by (\ref{first principles}),  $\Tr _kB^{\delta}=0$. Hence $B^{\delta}$ is algebraic over $k$ and is a domain,  and hence $B^{\delta}$ is a field. Consider a local slice $s$ of $\delta$ of $\delta$-degree $n$. By  $(1$) of (\ref{first principles}),  $D_n(x)\in B^{\delta}$ (which is a field) and hence $D_n(x)$ is a unit.  We can conclude that $s$ is a slice. Thus every non-trivial exponential map of $B$ has a slice,  and hence $\ML (B)=\MLs (B)$.
\end{proof}

The following result shows that a similar behavior as in (\ref{stab}) holds for $\MLs$.

\begin{proposition}\label{stab*}
    Let $B$ be a $k$-domain with $\Tr _kB=1$, where $k$ is an infinite field. Then $\MLs (B^{[n]})=\MLs (B)$ for every integer $n\geq 0$. Moreover,  we have that $\ML (B^{[n]})=\MLs (B^{[n]})$.
\end{proposition}

\begin{proof}
    By (\ref{stab}),  we have that $\ML (B^{[n]})=\ML (B)$. By (\ref{deg1}),  we also have that  $\MLs (B)=\ML (B)$. Thus,  $\MLs (B)=\ML (B^{[n]})$. Since $\ML (B^{[n]}) \subseteq \MLs (B^{[n]})$,  we get that $\MLs (B)\subseteq \MLs (B^{[n]})$. By $(8)$ of (\ref{first principles}),  we also have that $\MLs (B^{[n]})\subseteq \MLs (B)$ and hence $\MLs (B^{[n]})=\MLs (B)$. By (\ref{stab}) and (\ref{deg1}),  we have that $\ML (B^{[n]})=\MLs (B^{[n]})$.
\end{proof}

\begin{lemma}\label{equality}
    Let $B$ be an affine $k$-domain with $\Tr _kB=n\geq 2$. Suppose $\Tr _k\MLs (B)=n-1$. Then $\ML (B)=\MLs (B)$ and $B=C^{[1]}$ for some rigid subring $C$ of $B$. In particular,  $B$ is a semi-rigid domain.
\end{lemma}

\begin{proof}
    Given $\Tr _k\MLs (B)=n-1$, so there exists a $\delta \in \Exps (B)$. By Slice Theorem (\ref{slice theorem}),  $B=C^{[1]}$  where $C=B^{\delta}$. Notice that $\Tr _kC=n-1$ and  $\MLs (B)\subseteq C$. $\MLs (B)$ and $C$ are factorially closed in $B$ and have the same transcendence degree so it follows that $\MLs (B)=C$. By $(2)$ of (\ref{ML*}), $C$ is rigid  and  by Semi-Rigidity Theorem (\ref{semi}), $\ML (B)=C$ and that $B$ is semi-rigid.
\end{proof}

\begin{lemma}\label{tr2}
    Let $B$ be an affine $k$-domain with $\Tr _kB=2$. Suppose that $\MLs (B)\neq B$. Then $\MLs (B)=\ML (B)$.

\end{lemma}

\begin{proof}
        Since $\MLs (B)\neq B$, hence $\Tr _k\MLs (B)\leq 1$. If $\Tr _k\MLs (B)=1$, then by (\ref{equality}),  the result follows.  Suppose $\Tr _k\MLs (B)=0$, then since $\ML (B)\subseteq \MLs (B)$ we get that $\Tr _k\ML (B)=0$. $\ML (B)$ and $\MLs (B)$ are factorially closed in $B$, and of the same transcendence degree so it follows that $\ML (B)=\MLs (B)$. So in all cases, we get that $\ML(B)=\MLs (B)$. 
\end{proof}

\begin{corollary}
    Let $k$ be an algebraically closed field. Suppose $B$ is an affine $\Ufd$ over $k$ such that $\Tr _kB=2$. Then $\ML (B)=\MLs (B)$.
\end{corollary}

\begin{proof}
    If $B$ is rigid, then the result holds by definition. Suppose $B$ is not rigid,  then by (\ref{ufd slice}),  $B=(B^{\delta})^{[1]}$ for some non-trivial exponential map $\delta \in \Exp (B)$ and there exists an  $\epsilon \in \Exps (B)$ such that $B^{\epsilon}=B^{\delta}$. Thus $\Exps (B)\neq \phi$ and hence $\MLs (B)\neq B$  and by (\ref{tr2}),  the result follows. 
\end{proof}

\begin{proposition}\label{equality2}
    Let $B$ be a $\Ufd$ over a field $k$ with $\Tr _kB=2$, where $k$ is an infinite field. Suppose $B$ is geometrically factorial and $B$ is not rigid. Then $\MLs (B^{[n]})=\MLs (B)$ for every integer $n\geq 0$. Moreover, we have that $\ML (B^{[n]})=\MLs (B^{[n]})$.
\end{proposition}

\begin{proof}
        By (\ref{stab2}), we have that $\ML (B^{[n]})=\ML (B)$. By (\ref{gufd slice}), we have $\MLs (B)\neq B$. By (\ref{tr2}), we also have that  $\MLs (B)=\ML (B)$. Thus,  $\MLs (B)=\ML (B^{[n]})$. Since $\ML (B^{[n]}) \subseteq \MLs (B^{[n]})$, we get that $\MLs (B)\subseteq \MLs (B^{[n]})$. By $(8)$ of (\ref{first principles}), we also have that $\MLs (B^{[n]})\subseteq \MLs (B)$ and hence $\MLs (B^{[n]})=\MLs (B)$. By (\ref{stab2}) and (\ref{tr2}),  we have that $\ML (B^{[n]})=\MLs (B^{[n]})$.
\end{proof}

\begin{corollary}\label{ack}
    Let $B$ be a $\Ufd$ over an algebraically closed field $k$ with $\Tr _kB=2$ and $B$ is not rigid. Then  $\MLs (B^{[n]})=\MLs (B)$ for every integer $n\geq 0$. Moreover, we have that $\ML (B^{[n]})=\MLs (B^{[n]})$.    \end{corollary}

\begin{lemma}\label{coincide}
    Let $B$ be an affine $\Ufd$ over a field $k$ with $\Tr _kB=3$, where $k$ is an infinite field. Suppose $B$ is geometrically factorial. If $\MLs (B)\neq B$,  then $\ML (B)=\MLs (B)$.
\end{lemma}

\begin{proof}
    Since $\MLs (B)\neq B$,  we have that $\Tr _k\MLs (B)\leq 2$. 

    \emph{Case 1}: $\Tr _k\MLs (B)=2$. The result follows from (\ref{equality}). 

    \emph{Case 2}: $\Tr _k\MLs (B)\leq 1$. Let $\delta \in \Exps (B)$. By Slice Theorem (\ref{slice theorem}), we have $B=A^{[1]}$ where $A=B^{\delta}$ and $\Tr _kA=2$. Since $\Tr _k\MLs (B)\leq 1$, by $(2)$ of (\ref{ML*}), $A$ is not rigid. $B=A^{[1]}$ and $B$ is geometrically factorial implies that $A$ is  geometrically factorial with $\Tr _kA=2$. By (\ref{gufd slice}), $\MLs (A)\neq A$ and since $B=A^{[1]}$ it follows by (\ref{equality2}) that $\ML (B)=\MLs (B)$.
\end{proof}

\begin{corollary}\label{tr3 equality}
       Let $k$ be an algebraically closed field and $B$ be an affine $\Ufd$ over $k$ with $\Tr _kB=3$. If $\MLs (B)\neq B$,  then $\ML (B)=\MLs (B)$. 
\end{corollary}

\begin{remark}\label{Gaif}
   In \cite[Theorem $1$]{GS}, it is shown that when $k$ is an algebraically closed field of zero characteristic and if $\MLs (B)\neq B$,  then $\ML (B)=\MLs (B)$. We note that $(2)$ of (\ref{ML*}) and (\ref{stab*}) are identical to the properties of $\ML$-invariant, which leads to the following question. 
\end{remark}
\begin{question}\label{quest}
 Let $B$ be an affine $k$-domain where $k$ is an algebraically closed field of arbitrary characteristic. Suppose $\MLs (B)\neq B$. Is $\MLs (B)=\ML (B)?$
\end{question}
We note that (\ref{stab*}), (\ref{tr2}),  (\ref{equality2}) and (\ref{tr3 equality}) answer this question positively in some cases.\\

\noindent
{\bf Acknowledgement:} The author thanks Manoj K. Keshari for reviewing earlier drafts and suggesting improvements. The author thanks the referee for suggesting that the hypothesis in proposition (\ref{equality2}) and corollary (\ref{ack}) can be weakened from $\MLs (B)\neq B$ to $B$ is not rigid. The author is supported by the Prime Minister's Research Fellowship (PMRF), Government of India (ID: 1301165).  
\bibliographystyle{abbrv}
\bibliography{refs} 
\end{document}